\documentclass[preprint,12pt]{elsarticle}

\usepackage{amssymb}

\newtheorem{thm}{Theorem}

\newdefinition{dfn}{Definition}
\newdefinition{ex}{Example}
\newproof{proof}{Proof}
\newtheorem{cor}{Corollary}

\begin{document}

\begin{frontmatter}



\title{Asymptotic formulas for eigenvalues and eigenfunctions  of boundary value problem}


\author[rvt]{O. Sh. Mukhtarov\corref{cor1}}
\ead{omukhtarov@yahoo.com} \cortext[cor1]{Corresponding Author (Tel:
+90 356 252 16 16, Fax: +90 356 252 15 85)}
\author[rvt]{K. Aydemir}
\ead{kadriye.aydemir@gop.edu.tr}

\address[rvt]{Department of Mathematics, Faculty of Arts and Science, Gaziosmanpa\c{s}a University,\\
 60250 Tokat, Turkey}

\begin{abstract}
In this paper we are concerned with a new class of BVP' s consisting
of eigendependent boundary conditions and two supplementary
transmission conditions at one interior point. By modifying some
techniques of classical Sturm-Liouville theory
 and  suggesting own approaches we find asymptotic formulas for the eigenvalues and eigenfunction.
\end{abstract}

\begin{keyword}
Sturm-Liouville problems, eigenvalue, eigenfunction, asymptotics of
eigenvalues and eigenfunction.


\end{keyword}

\end{frontmatter}


\section{Introduction}
 Many topics in
mathematical physics require investigations of eigenvalues and
eigenfunctions of boundary value problems.  These investigations are
of utmost importance for theoretical and applied problems in
mechanics, the theory of vibrations and stability, hydrodynamics,
elasticity, acoustics, electrodynamics, quantum mechanics, theory of
systems and their optimization, theory of random processes, and many
other branches of natural science. Such problems are formulated in
many different ways.

In this study we shall investigate a new class of Sturm-Liouville
type problem which consist of a Sturm-Liouville equation contained
\begin{equation}\label{1}
\Gamma (y):=-y^{\prime \prime }(x,\lambda)+ q(x)y(x,\lambda)=\lambda
y(x,\lambda)
\end{equation}
to hold in finite interval $(-\pi, \pi) $ except at one inner point
$0 \in (-\pi, \pi) $ , where discontinuity in $u  \ \textrm{and} \
u'$ are prescribed by transmission conditions
\begin{equation}\label{4}
\Gamma_{1}(y):=a_{1}y'(0-,\lambda)+a_{2}y(0-,\lambda)+a_{3}y'(0+,\lambda)+a_{4}y(0+,\lambda)=0,
\end{equation}
\begin{equation}\label{5}
\Gamma_{2}(y):=b_{1}y'(0-,\lambda)+b_{2}y(0-,\lambda)+b_{3}y'(0+,\lambda)+b_{4}y(0+,\lambda)=0,
\end{equation}
together with the  boundary conditions
\begin{equation}\label{2}
 \Gamma_{3}(y):=\cos \alpha y(-\pi,\lambda)+\sin\alpha y'(-\pi,\lambda)=0,
\end{equation}
\begin{equation}\label{3}
\Gamma_{4}(y):=\cos\beta y(\pi,\lambda)+\sin\beta y'(\pi,\lambda)=0.
\end{equation}
 We describe some analytical solutions of the
problem and find asymptotic formulas of eigenvalues and
eigenfunctions. These boundary conditions are of great importance
for theoretical and applied studies and have a definite mechanical
or physical meaning (for instance, of free ends). Also the problems
with transmission conditions  arise in mechanics, such as thermal
conduction problems for a thin laminated plate, which studied in
\cite{tik}. This class of problems essentially differs from the
classical case, and its investigation requires a specific approach
based on the method of separation of variables. Moreover the
eigenvalue parameter appear in one of the boundary conditions and
two new conditions added to boundary conditions called transmission
conditions.
\section{\textbf{The fundamental solutions and characteristic
Function } }
 Let  T=$ \left[%
\begin{array}{cccc}
  \beta^{-}_{10} & \beta^{-}_{11} & \beta^{+}_{10} & \beta^{+}_{11} \\
  \beta^{-}_{20} & \beta^{-}_{21} & \beta^{+}_{20} & \beta^{+}_{21}
  \\
\end{array} %
 \right]. $
 Denote  the determinant of the determinant of the k-th
 and
i-th columns of the matrix T  by $\rho_{kj}$. Note that throughout
this study we shall assume that $ \rho_{12}>0 \ \ \textrm{and} \
\rho_{34}>0.$
 With a view to constructing the
characteristic function we define  two solution $\phi(x,\lambda) \
\textrm{and}  \ \chi(x,\lambda)$ as follows. Denote the solutions of
the equation (\ref{1}) satisfying the initial conditions
\begin{equation}\label{7}
y(-\pi)=\sin\alpha, \ y^{\prime }(-\pi)=-\cos\alpha
\end{equation}
and
\begin{equation}\label{10}
y(\pi)=-\sin\beta, \ y^{\prime }(\pi)=\cos\beta
\end{equation}
by $u=\phi_1(x,\lambda)$ and $u=\chi_{2}(x,s),$ respectively. It is
known that the initial-value problems has an unique solutions \
$u=\phi_1(x,\lambda)$ and $u=\chi_{2}(x,\lambda),$ \ which is an
entire function of $s \in \mathbb{C}$ for each fixed $x \in
[-\pi,0)$  and $x \in (0,\pi]$ respectively. (see, for example,
\cite{titc}). Using this solutions we can prove that the equation
(\ref{1}) on $ [-\pi,0)$ and $x \in (0,\pi]$ has solutions
$u=\varphi^{+}(x,s)$ and $u=\psi^{-}(x,s),$ satisfying the initial
conditions
\begin{eqnarray}\label{8}
&&y(0) =\frac{1}{\rho_{12}}(\rho_{23}\phi_1(0,\lambda)
+\rho_{24}\frac{\partial\phi_1(0,\lambda)}{\partial x})\\
&& \label{9} y^{\prime }(0)
=\frac{-1}{\rho_{12}}(\rho_{13}\phi_1(0,\lambda)
+\rho_{14}\frac{\partial\phi_1(0,\lambda)}{\partial x}).
\end{eqnarray}
and
\begin{eqnarray}\label{11}
&&y(0) =\frac{-1}{\rho_{34}}(\rho_{14}\chi_{2}(0,\lambda
)+\rho_{24}\frac{\partial\chi_{2}(0,\lambda )}{\partial x}),\\
&& \label{12} y^{\prime }(0)
=\frac{1}{\rho_{34}}(\rho_{13}\chi_{2}(0,\lambda
)+\rho_{23}\frac{\partial\chi_{2}(0,\lambda )}{\partial x}).
\end{eqnarray}
respectively.
\section{ Some asymptotic approximation formulas
for fundamental solutions}
  By applying the method of
variation of parameters we can prove that the next integral and
integro-differential equations are hold for $k=0$ and $k=1.$
\begin{eqnarray}\label{(4.2)}
\frac{d^{k}}{dx^{k}}\phi_1(x,\lambda) &=&\sin\alpha
\frac{d^{k}}{dx^{k}}\cos \left[s\left( x+\pi\right)\right]-
\frac{\cos\alpha}{s}\frac{d^{k}}{dx^{k}}\sin \left[s\left( x+\pi\right)\right]\nonumber \\
&&+ \frac{1}{s}\int\limits_{-\pi}^{x}\frac{d^{k}}{dx^{k}}\sin
\left[s\left( x-z\right)\right] q(z)\phi_1(z,\lambda) dz
\end{eqnarray}
\begin{eqnarray}
\frac{d^{k}}{dx^{k}}\chi_1(x,\lambda )
&=&-\frac{1}{\rho_{34}}(\rho_{14}\chi _2(0,\lambda )+\rho_{24}
\frac{\partial\chi_2(0,\lambda )}{\partial x})
\frac{d^{k}}{dx^{k}}\cos \left[
s x\right] \nonumber \\
&&+\frac{1}{s \rho_{34}}(\rho_{13}\phi_2(0,\lambda
)+\rho_{23}\frac{\partial\chi_2(0,\lambda )}{\partial
x})\frac{d^{k}}{dx^{k}}\sin \left[s x\right]  \nonumber  \\
&&+\frac{1}{s}\int\limits_{x}^{0}\frac{d^{k}}{dx^{k}}\sin \left[
s\left( x-z\right) \right] q(z)\chi_2(z,\lambda )dz \label{(4.22)}
\end{eqnarray}
for $x \in [-\pi,0)$ and
\begin{eqnarray}
\frac{d^{k}}{dx^{k}}\phi_2(x,\lambda)
&=&\frac{1}{\rho_{12}}(\rho_{23}\phi_1(0,\lambda)
+\rho_{24}\frac{\partial\phi_1(0,\lambda)}{\partial x})
\frac{d^{k}}{dx^{k}}\cos \left[
s x\right]  \nonumber  \\
&&-\frac{1}{s
\rho_{12}}(\rho_{13}\phi_1(0,\lambda)+\rho_{14}\frac{\partial\phi_1(0,\lambda)}{\partial
x})\frac{d^{k}}{dx^{k}}\sin \left[s x \right]  \nonumber \\
&&+\frac{1}{s}\int\limits_{0}^{x}\frac{d^{k}}{dx^{k}}\sin\left[s\left(
x-z\right) \right] q(z)\phi_2(z,\lambda)dz \label{(4.a)}
\end{eqnarray}
\begin{eqnarray}
\frac{d^{k}}{dx^{k}}\chi_{2}(x,\lambda )&=&-\sin \beta
\frac{d^{k}}{dx^{k}} \cos \left[s\left(
\pi-x\right)\right]-\frac{\cos \beta}{s}\frac{d^{k}}{dx^{k}}\sin
\left[ s\left(\pi-x\right) \right] \nonumber\\
&&+\frac{1}{s}\int\limits_{x}^{\pi}\frac{d^{k}}{dx^{k}}\sin
\left[s\left(x-z\right)\right] q(z)\chi_{2}(z,\lambda)dz
\label{(4.21)}
\end{eqnarray}
for $x \in (0,\pi]$. Now we are ready to prove the following
theorems.
\begin{thm} \label{(4.n)}
Let $\lambda =s^{2}$, $Ims=t.$ Then \ if $\sin\alpha \neq 0$
\begin{eqnarray}
\frac{d^{k}}{dx^{k}}\phi _{1\lambda}(x ) &=&\sin\alpha \frac{d^{k}}{dx^{k}}%
\cos \left[ s\left( x+\pi\right) \right] +O\left( \left| s\right|
^{k-1}e^{_{\left| t\right| (x+\pi)}}\right)
\label{(4.3)} \\
\frac{d^{k}}{dx^{k}}\phi _{2\lambda}(x )
&=&\frac{\rho_{24}}{\rho_{12}}\sin\alpha s \sin \left[s\pi\right]
\cos\left[ s x\right]  +O\left(|s| ^{k} e^{\left|
t\right|(x+\pi)}\right) \label{(4.4)}
\end{eqnarray}
as $\left| \lambda \right| \rightarrow \infty $, while if $\sin\alpha =0$%
\begin{eqnarray}
\frac{d^{k}}{dx^{k}}\phi _{1\lambda}(x ) &=&-\frac{\cos\alpha}{s}%
\frac{d^{k}}{dx^{k}}\sin \left[ s(x+\pi)\right] +O\left( \left|
s\right| ^{k-2}e^{\left| t\right| (x+1)}\right) \label{(4.5)}
\\
\frac{d^{k}}{dx^{k}}\phi _{2\lambda}(x )
&=&-\frac{\rho_{24}}{\rho_{12}}\cos\alpha \cos \left[s\pi\right]
\cos\left[ s x\right]  +O\left(|s| ^{k-1} e^{\left|
t\right|(x+\pi)}\right)   \label{(4.6)}
\end{eqnarray}
as $\left| \lambda \right| \rightarrow \infty $ ($k=0,1)$. Each of
this asymptotic equalities hold uniformly for $x.$
\end{thm}
\begin{proof}

\end{proof}
Similarly  we can easily obtain the following Theorem for
$\chi_{i}(x,\lambda) (i=1,2).$
\begin{thm} \label{(c1)}
Let $\lambda =s^{2}$, $Ims=t.$ Then \ if $\sin\beta \neq 0$
\begin{eqnarray}
\frac{d^{k}}{dx^{k}}\chi _{2\lambda}(x ) &=&\sin\beta \frac{d^{k}}{dx^{k}}%
\cos \left[ s\left( \pi-x\right) \right] +O\left( \left| s\right|
^{k-1}e^{_{\left| t\right| (\pi-x)}}\right)
\label{(c2)} \\
\frac{d^{k}}{dx^{k}}\chi _{1\lambda}(x ) &=&-
\frac{\rho_{24}}{\rho_{34}}\sin\beta s \sin \left[s\pi\right]
\cos\left[ s x\right]  +O\left(|s| ^{k} e^{\left|
t\right|(\pi-x)}\right) \label{(4.n1)}
\end{eqnarray}
as $\left| \lambda \right| \rightarrow \infty $, while if $\sin\beta=0$%

\begin{eqnarray}
\frac{d^{k}}{dx^{k}}\chi _{2\lambda}(x ) &=&-\frac{\cos\beta }{s}%
\frac{d^{k}}{dx^{k}}\sin \left[ s(\pi-x)\right] +O\left( \left|
s\right| ^{k-2}e^{\left| t\right| (\pi-x)}\right) \label{(c3)}
\\
\frac{d^{k}}{dx^{k}}\chi _{1\lambda}(x) &=&-
\frac{\rho_{24}}{\rho_{34}}\cos\beta\cos \left[s\pi\right] \cos
\left[ s x \right]  +O\left(|s| ^{k-1} e^{\left|
t\right|(\pi-x)}\right)
\end{eqnarray}
as $\left| \lambda \right| \rightarrow \infty $ ($k=0,1)$. Each of
this asymptotic equalities hold uniformly for $x.$
\end{thm}
\section{Asymptotic behaviour
 of eigenvalues and corresponding eigenfunctions%
} It is well-known from ordinary differential equation theory that
the Wronskians $W[\phi _{1\lambda},\chi _{1\lambda}]_{x}$ and
$W[\phi _{2\lambda}, \chi _{2\lambda}]_{x}$ are independent of
variable $x.$ By using (\ref{4}) and (\ref{5}) we have
\begin{eqnarray*} \label{7} w_{1}(\lambda) &=& \phi_{1}(0,\lambda)\frac{\partial
\chi_{1}(0,\lambda)}{\partial x}-\chi_{1}(0,\lambda)\frac{\partial
\phi_{1}(0,\lambda)}{\partial x} \nonumber \\&=&
\frac{\rho_{12}}{\rho_{34}}(\phi_{2}(0,\lambda) \frac{\partial
\chi_{2}(0,\lambda)}{\partial x}-\chi_{2}(0,\lambda)\frac{\partial
\phi_{2}(0,\lambda)}{\partial x})\nonumber \\&=&
\frac{\rho_{12}}{\rho_{34}}w_{2}(\lambda)
\end{eqnarray*}

for each $\lambda \in \mathbb{C}$. It is convenient to introduce the
notation
 \begin{equation}\label{8}w(\lambda):=\rho_{34} w_{1}(\lambda) = \rho_{12}
w_{2}(\lambda).\end{equation}

Now by modifying the standard method we  prove that all eigenvalues
of the problem $(\ref{1})-(\ref{3})$ are real.
\begin{thm}
The  eigenvalues of the boundary-value-transmission problem
$(\ref{1})-(\ref{3})$ are real.
\end{thm}
\begin{proof}

\end{proof}
\begin{cor}\label{cor2}Let $u(x)$ \textrm{and} $v(x)$   be eigenfunctions corresponding to distinct eigenvalues.
Then they are orthogonal in the sense of the following equality
\begin{eqnarray}\label{2.3}\ \rho_{12}\int_{-\pi}^{0} u(x)v(x)dx + \rho_{34} \int_{0}^{\pi}
u(x)v(x)dx=0.
\end{eqnarray}
\end{cor}
Since the Wronskians of $\phi _{2\lambda}(x )$ and $\chi
_{2\lambda}(x )$ are independent of $x$, in particular, by putting
$x=1$ we have
\begin{eqnarray}\label{(ko)}
w(\lambda ) &=&\phi _{2}(\pi,\lambda )\chi_{2}^{\prime }(\pi,\lambda
)-\phi
_{2}^{\prime }(\pi,\lambda )\chi _{2}(\pi,\lambda ) \nonumber\\
&=&\cos \beta \phi _{2}(\pi,\lambda )+\sin\beta\phi _{2}^{\prime
}(\pi,\lambda ).
\end{eqnarray}
 Let $\lambda =s^{2}$, $Ims=t.$ By substituting $(\ref{(c2)})$ and
$(\ref{(c3)})$ in $(\ref{(ko)})$ we obtain easily the following
asymptotic representations\\ \textbf{(i)} If $\sin\beta \neq 0$ and
$\sin\alpha\neq 0$, then
\begin{equation}
w(\lambda )=-\frac{\rho_{24}}{\rho_{12}} \sin\alpha\sin\beta \ s^{2}
\sin^{2} \left[s\pi\right] +O\left( \left| s\right| e^{2\pi\left| t
\right| }\right) \label{(4.15)}
\end{equation}
\textbf{(ii)} If $\sin\beta\neq  0$ and $\sin\alpha= 0$, then
\begin{equation}
w(\lambda )=\frac{\rho_{24}}{\rho_{12}}\cos\alpha \sin\beta  \ s\cos
\left[s\pi\right] \sin \left[s\pi\right] +O\left( e^{2\pi\left| t
\right| }\right)  \label{(4.16)}
\end{equation}
\textbf{(iii)} If $\sin\beta= 0$ and $\sin\alpha\neq 0$, then
\begin{equation}
w(\lambda )=\frac{\rho_{24}}{\rho_{12}}\sin\alpha \cos\beta\ s\sin
\left[s\pi\right] \cos \left[s\pi\right] +O\left(  e^{2\pi\left| t
\right| }\right)  \label{(4.17)}
\end{equation}
\textbf{(iv)} If $\sin\beta=0$ and $\sin\alpha= 0$, then
\begin{equation}
w(\lambda )=-\frac{\rho_{24}}{\rho_{12}}\cos\beta \cos\alpha
\cos^{2}\left[s\pi\right] +O\left( \frac{1}{\left| s\right|}
e^{2\pi\left| t \right| }\right) \label{(4.18)}
\end{equation}
Now we are ready to derived the needed asymptotic formulas for
eigenvalues and  eigenfunctions.
\begin{thm}
The boundary-value-transmission problem $(\ref{1})$-$(\ref{3})$ has
an precisely numerable many real eigenvalues, whose behavior may
be expressed by $\left\{ \lambda _{n,1}\right\} $ with following asymptotic as $n\rightarrow \infty $%

\textbf{(i)} If   $\sin\beta \neq 0$ and $\sin\alpha\neq 0$,then
\begin{equation}
s_{n,1}= (n-\frac{1}{2})+O\left( \frac{1}{n}\right) \label{(5.1)}
\end{equation}
\textbf{(ii)} If $\sin\beta\neq  0$ and $\sin\alpha= 0$, then
\begin{equation}
s_{n,1}=\frac{n}{2} +O\left( \frac{1}{n}\right), \label{(5.2)}
\end{equation}
\textbf{(iii)} If $\sin\beta= 0$ and $\sin\alpha\neq 0$, then
\begin{equation}
s_{n,1}=\frac{n}{2}+O\left( \frac{1}{n}\right) , \label{(5.3)}
\end{equation}
\textbf{(iv)} If $\sin\beta=0$ and $\sin\alpha= 0$, then
\begin{equation}
s_{n,1}=\frac{n}{2} +O\left( \frac{1}{n}\right) , \label{(5.4)}
\end{equation}
where $\lambda _{n,1}=s_{n,1}^{2}$ $\label{t4}$.
\end{thm}

\begin{proof}

\end{proof}
Using this asymptotic expression of eigenvalues we can easily obtain
the corresponding asymptotic expressions for eigenfunctions of the
problem $(\ref{1})$-$(\ref{3})$.
 Recalling that $\phi_{s _{n}}(x)$ is an eigenfunction
 according to the eigenvalue $s _{n},$ and by putting (\ref{(5.1)}) in the (\ref{(4.3)})-(\ref{(4.4)}) for $k=0,1$
 and denoting the corresponding  eigenfunction as

\begin{thm}\label{(4.n)}
\textbf{(i)} If   $\sin\beta \neq 0$ and $\sin\alpha\neq 0$,then
\begin{eqnarray*}
\phi_{s _{n}}(x) &=&\sin\alpha \cos \left[ (n-\frac{1}{2})\left(
x+\pi\right) \right] +O\left( \frac{1}{n}\right) \label{04}
\end{eqnarray*}
\textbf{(ii)} If $\sin\beta\neq  0$ and $\sin\alpha= 0$, then
\begin{eqnarray*}
 \phi_{s _{n}}(x) &=&\frac{\rho_{24}}{\rho_{12}}\sin\alpha
\frac{n}{2} \sin \left[\frac{n}{2}\pi\right] \cos\left[ \frac{n}{2}
x\right] +O\left(1\right) \label{(o1)}
\end{eqnarray*}
\textbf{(iii)} If $\sin\beta= 0$ and $\sin\alpha\neq 0$, then
\begin{eqnarray*}
\phi_{s _{n}}(x) &=&-\frac{2\cos\alpha}{n}%
\sin \left[ \frac{n}{2}(x+\pi)\right] +O\left( \frac{1}{n^2}\right)
\label{(02)}
\end{eqnarray*}
\textbf{(ii)} If $\sin\beta=0$ and $\sin\alpha= 0$, then
\begin{eqnarray*}
\phi_{s _{n}}(x) &=&-\frac{\rho_{24}}{\rho_{12}}\cos\alpha \cos
\left[\frac{n}{2}\pi\right] \cos\left[ \frac{n}{2} x\right]
+O\left(\frac{1}{n}\right)   \label{03}
\end{eqnarray*}
as $\left| \lambda \right| \rightarrow \infty $ ($k=0,1)$. Each of
this asymptotic equalities hold uniformly for $x.$
\end{thm}

\end{document}